\documentclass[11pt]{amsart}
\usepackage{graphicx}
\usepackage{amssymb}
\usepackage{cite}
\usepackage{hyperref}
\usepackage{enumerate}

\usepackage[usenames,dvipsnames]{color}

\numberwithin{equation}{section} 

\newtheorem{theorem}{Theorem}[section]
\newtheorem{theorem*}{Theorem}
\newtheorem{lemma}[theorem]{Lemma}
\newtheorem{proposition}[theorem]{Proposition}
\newtheorem{corollary}[theorem]{Corollary}

\newtheorem{remark}[theorem]{Remark}
\newtheorem{remark*}[theorem*]{Remark}
\newtheorem{definition}[theorem]{Definition}
\newtheorem{question}[theorem]{Question}
\newtheorem{assumption}[theorem]{Assumption}
\newtheorem{conjecture}[theorem]{Conjecture}

\def\R{{\mathbb R}}

\def\l{\lambda}

\def\1{\left(}
\def\2{\right)}
\def\3{\left\{}
\def\4{\right\}}
\def\8{\infty}

\DeclareMathOperator*{\spt}{spt}
\DeclareMathOperator*{\dist}{dist}

\DeclareMathOperator*{\distb}{Dist_b}

\newcommand\vol{{\rm vol}}

\begin{document}
\title{Dynamics of Optimal Partial Transport}

\author[G. D\'avila]{Gonzalo D\'avila}
\address{Departamento de Matem\'atica, Universidad T\'ecnica Federico Santa Mar\'ia, Avenida Espa\~na 1680, Valpara\'iso, Chile}
\email{gonzalo.davila@usm.cl}

\author[Y-H. Kim]{Young-Heon Kim}
\address{Department of Mathematics\\ University of British Columbia\\ Vancouver, V6T 1Z2 Canada}
\email{yhkim@math.ubc.ca}

\thanks{This research is partially supported by 
Natural Sciences and Engineering Research Council of Canada Discovery Grants 371642-09 and 2014-05448 as well as the  Alfred P. Sloan Research Fellowship 2012--2016.
Part of this research has been done while Y.-H.K. was visiting 
 Korea Advanced Institute of Science and Technology (KAIST), 
the Mathematical Sciences Research Institute (MSRI) for the thematic program ``Optimal Transport: Geometry and Dynamics", and 
the Fields institute, Toronto for the thematic program on ``Calculus of Variations''.
\copyright 2015 by the authors.
}

\begin{abstract}
Optimal partial transport, which was initially studied by Caffarelli and McCann \cite{Ca-Mc}, is a variant of optimal transport theory, where only a portion of mass is to be transported in an efficient way.  Free boundaries naturally arise as the boundary of the region where the actual transport occurs.  This paper considers the  evolution dynamics of the free boundaries in terms of the change of $m$, the allowed amount of transported mass 
or the change of $\lambda$, the transportation cost cap, i.e. the allowed maximum cost for a unit mass to be transported. Focusing on the quadratic cost function, we show H\"older and Lipschitz estimates on the  speed of the free boundary motion in terms of $m$ and $\lambda$, respectively. It is also shown that the parameter $m$ is a Lipschitz function of $\lambda$, which previously was known only to be a continuous increasing function \cite{Ca-Mc}. 
\end{abstract}
\maketitle

\maketitle


\section{Introduction}

 Given two mass distributions we consider the phenomena of matching them together in a cost efficient way.  We are interested in the case when only a fraction of the mass is to be matched  and therefore only a part of the mass distributions are transported.  This, so-called, {\em optimal partial transport} problem, has been an interest among researchers starting from the work of Caffarelli and McCann \cite{Ca-Mc}. It is a natural generalization of the optimal transport problem of Monge and Kantorovich \cite{Mon81, Kan04b} where the full masses are matched; see, \cite{Vil03, Vil09} for a modern survey. 

In this paper, our aim is to understand the dynamical behaviour of the solution to the optimal partial transport problem: First, one can formulate this partial transport problem in two equivalent forms, using  two different parameters. One is the amount $m$ of mass to be transported, and the other one is the transportation cost cap $\lambda$, the allowed maximum cost for a unit mass to be transported. See Section~\ref{S: two formulations} for more details for the two equivalent formulations of the partial transport problem. 
As these parameters increase, the {\em active region} where the transport of mass actually occurs, changes (in fact, it changes monotonically \cite{Ca-Mc}).  Focusing on the cost function $c(x,y) =\frac 1 2  |x-y|^2$, the main goal of this paper is to present certain quantitative estimates (see Theorem~\ref{th: main} for precise statements) on  the change of the active region with respect to the change of the parameters, $m$ or $\lambda$.  In fact, such an estimate with respect to $m$ is not difficult, since the difference in amount of mass can be easily related to the volume change of the active region. Estimates with respect to $\lambda$  are more difficult to obtain, since  the relation between the transportation cost cap and the amount of transported mass, thus with the volume of active region, is indirect.  Note that the parameter $m$ can be considered as a function of $\lambda$, and it is known to be monotone \cite{Ca-Mc}. As a byproduct of our estimates, we show that $\lambda \mapsto m(\lambda)$ is Lipschitz (see Theorem~\ref{th: main} item 4, and see also, Corollary~\ref{cor: Lip in lambda}). Our main technical tool is a certain monotonicity of the potential functions associated to the partial transport problem (see Theorem~\ref{th: monotonicity}).  

Even though we do not pursue it here, these results can be extended to more general class of examples (see Remark~\ref{rm: conditions on c}).  However, to show a strict separation of the free boundaries for different values of parameters, we use the special structure of $c(x, y) =\frac 1 2  |x-y|^2$, and the assumption that mass distributions are smooth: we show a version of strong maximum principle  for the Monge-Amp\'ere equation, from which we show that the active region grows strictly monotonically in a point-wise sense (see Section~\ref{S: strict mono},  Theorem~\ref{th: strict free}).

\subsection*{Organization of the paper:}
Section~\ref{S: two formulations} gives the notation throughout the paper, and explains the equivalence between the two approaches to partial transport problems, with respect to $m$ and $\lambda$, respectively. Section~\ref{S: pre} sets up the main assumptions, and explains preliminary results.  Section~\ref{S: main} explains the main results, whose proofs are given in Sections~\ref{S: mass and free} and \ref{S: cost and free},  
where the key monotonicity of potential functions is given in Section~\ref{S: monotone potential}. 
Finally, Section~\ref{S: strict mono} shows the strict monotonicity of the active regions under additional assumptions.

\subsection*{Acknowledgement:} We thank Inwon Kim for helpful discussions and interest in this work.

\section{Two formulations of optimal partial transport}\label{S: two formulations}

In this section we give a more precise statement of the optimal partial transport problem that was initiated by Caffarelli and McCann \cite{Ca-Mc}. We will provide two equivalent formulations.

\subsection{Notation}

Optimal transport problems consist of three basic ingredients: source and target distributions, cost, and transport plans. We explain below the necessary notation for these ingredients we will use in this paper. 

\begin{enumerate}[\bf 1.]
\item  {\bf Source and target mass distributions:} Let $f, g\in L^1(\R^n)$ be two nonnegative functions, with 
\begin{align}\label{eq: prob}
\int_{\R^n} f (x)dx = 1= \int_{\R^n} g(y) dy.  
\end{align} 
For simplicity, we will also assume that both $f$ and $g$ have finite second moments, namely,  $\int_{\R^n} |x|^2 f(x) dx , \int_{\R^n} |y|^2 g(y) dy < \infty$. 
\item {\bf Transport plans:}
\begin{enumerate}[\bf a.]
\item  Let $\Gamma(f,g)$ be the set of all Borel measures on $\R^n\times\R^n$ whose left and right marginals are $f$ and $g$ respectively, that is $\gamma\in\Gamma(f,g)$ if $\gamma$ satisfies
\begin{align*}
\gamma(A\times\R^n)=\int_A f(x)dx, \quad
\gamma(\R^n\times A)=\int_A g(y)dy
\end{align*}
for all $A\subset\R^n$ Borel.
 \item Let $\Gamma_{\leq}(f,g)$ be the set of all Borel measures on $\R^n\times\R^n$ whose left and right marginals are dominated by $f$ and $g$ respectively, that is $\gamma\in\Gamma_{\leq}(f,g)$ if $\gamma$ satisfies
\begin{align*}
\gamma(A\times\R^n)\leq\int_A f(x)dx,\quad
\gamma(\R^n\times A) \leq\int_A g(y)dy
\end{align*}
for all $A\subset\R^n$ Borel.
\end{enumerate}
 \item 
 Define now the total mass of a Borel measure $\gamma$ on $\R^n \times \R^n$ by
\[
\mathbf{m}(\gamma)=\int_{\R^n\times\R^n}d\gamma,
\]

\item {\bf Costs and modified costs from the transportation cost cap:}
\begin{enumerate}[\bf a.] 
\item Let  $c: \R^n \times \R^n \to \R$ be the transportation cost. 
\item Fix $0 \le \lambda < \infty$. We use this value $\lambda$ as a cap on the value of cost that we allow for transportation.   This leads us to define the following modified transportation cost.
\begin{align}
 c_\lambda(x, y) = \begin{cases}
    c(x,y)   & \text{if $c(x,y) < \l$}, \\
     \l  & \text{otherwise}.
\end{cases}
\end{align}
\end{enumerate}

\end{enumerate}
 
 \subsection{Two formulations}
 Now, we  describe the two equivalent formulations of the partial transport problem. 
\subsubsection{\bf $m$-problem:  changing the portion of mass}

  The optimal partial transport problem can be stated as follows \cite{Ca-Mc}:
and fix $0<m\le \min\{\|f\|_{L^1},\|g\|_{L^1}\}$ which is the amount of mass to be transported. Given a cost function $c:\R^n\times\R^n\to\R$ the associated cost functional is given by,
\[
\mathcal{C}(\gamma)=\int_{\R^n\times\R^n} c(x,y)d\gamma(x,y).
\]
The partial optimal transport problem consists then to minimize $\mathcal{C}(\gamma)$ among all possible $\gamma\in\Gamma_{\leq}(f,g)$ subject to added constraint of $\mathbf{m}(\gamma)=m$, that is
\begin{align}\label{eq: OPT}
 \inf_{\gamma \in \Gamma_{\le} (f, g) \& \, \mathbf{m}(\gamma) = m}  \mathcal{C} (\gamma).
\end{align}
This problem can also be formulated in the following way: Let 
\begin{align*}
 K_f(m) =\{ h \in L^1  \ | 0 \le  h \le f \text{ a.e.}  \quad \& \,  \| h \|_{L^1} = m\}. 
\end{align*}
Then, 
\begin{align*}
 \inf_{\gamma \in \Gamma_{\le} (f, g) \& \, \mathbf{m}(\gamma) = m}  \mathcal{C} (\gamma)
 = \inf_{\gamma \in \Gamma(h_1, h_2) \& h_1 \in K_f(m), h_2 \in K_g (m)} \mathcal{C}(\gamma)
\end{align*}
In particular,  when $c(x, y) = \frac 1 2 |x-y|^2$, the problem on the right corresponds to finding the minimum Wasserstein $W_2$ distance between the two sets $K_f(m)$ and $K_g(m)$. This point of view is borrowed from a recent work of DePhilippis, M\'esz\'aros, F. Santambrogio, Velichkov \cite{DePhSMV14}, where they considered a projection problem of a probability measure to the set $K_f$ with respect to the $W_2$ metric.

 Existence of the optimizer of \eqref{eq: OPT} follows easily from a compactness argument. 
 Uniqueness of the solution to the optimal partial transport problem has been established under reasonable conditions.  Also, a progress has been made in understanding regularity of the free boundaries; \cite{Ca-Mc, Fi09, Fi10, Chen-Indrei15}, especially when the cost function $c(x,y)$ is given by the distance squared $c(x,y)=\frac 1 2 \dist^2(x, y)$. 
 Kitagawa and Pass \cite{Kit-Pa-14p} then considered the problem in the case where there are finitely many mass distributions to be partially matched, and made a connection to the barycenter problem in the space of probability measures that was considered by Agueh and Carlier \cite{Agueh-Carlier11}. 
 
\begin{remark}
 We note that the barycenter problem in the space of probability measures was extended in \cite{Kim-Pass14p} to a continuous family of measures over Riemannian manifolds, finding Jensen type inequalities over the space of probability measures; the corresponding partial transport problem has not been considered yet. 
\end{remark}

 \subsubsection{\bf  $\lambda$-problem: changing the transportation cost cap: (pay flat rate or do not transfer after a certain threshold).}

As shown in  \cite{Ca-Mc}  the $m$-problem \eqref{eq: OPT} can be equivalently formulated using a Lagrange multiplier $\l\geq 0$ 
and considering the unrestricted minimization problem
\begin{align}\label{minprob}
C_\l(f,g)=\inf\limits_{\gamma\in\Gamma_{\leq}(f,g)} \int_{\R^n\times\R^n}(c(x,y)-\l)d\gamma(x,y).  %
\end{align}
Given $m$, there is  $\l$, so that the solutions to these problems coincide \cite{Ca-Mc}. 

We now observe that the value of $\l$ carries an economical meaning. First, consider
\begin{align}\label{eq: lambda problem}
 \inf_{\gamma \in \Gamma(f,g)} \int_{\R^n \times \R^n} c_\lambda (x, y) d\gamma (x, y) .
\end{align}
Note that here the admissible set is $\Gamma(f,g)$, not $\Gamma_{\le } (f, g)$. 
The cost $c_\l$ reflects the practical situation where for example, a taxi driver charges a flat rate after a given upper bound. The following simple observation explains the equivalence between \eqref{minprob} and \eqref{eq: lambda problem}, thus equivalence between the $m$-problem   \eqref{eq: OPT} and the $\lambda$-problem \eqref{eq: lambda problem}.   
\begin{proposition}[Equivalence of two $\lambda$ problems]\label{prop: lambda equiv} The two problems \eqref{minprob} and \eqref{eq: lambda problem} are equivalent, namely, 
 \begin{align}\label{eq: C lambda equivalence}
 &\inf\limits_{\gamma\in\Gamma_{\leq}(f,g)} \int_{\R^n\times\R^n}(c(x,y)-\l)d\gamma(x,y) \\\nonumber
& =  -\l + \inf_{\gamma' \in \Gamma(f,g)} \int_{\R^n \times \R^n} c_\lambda (x, y) d\gamma' (x, y) .
\end{align}
Moreover, a minimizer of  \eqref{minprob} is the restriction of a minimizer of \eqref{eq: lambda problem} on the set $\{ (x, y) \in \R^n \times \R^n \  |\  c(x, y) < \lambda\}$. 
\end{proposition}
\begin{proof}
 For each $\gamma' \in \Gamma(f, g)$, let $\gamma \in \Gamma_{\le}(f, g)$ be its restriction to the set  $
 \{ (x, y) \in \R^n \times \R^n \  |\  c(x, y) < \lambda\}$. 
 Now, consider
\begin{align*}
 \int_{\R^n \times \R^n} c_\lambda (x, y) d\gamma' (x, y) 
 & =  \int_{\R^n \times \R^n} c(x, y) d\gamma (x, y) + \l \int_{\R^n \times \R^n} d (\gamma' - \gamma)(x, y)\\
 & =  \int_{\R^n \times \R^n}( c(x, y) -\l)d\gamma (x, y) + \l \int_{\R^n \times \R^n} d \gamma' (x, y)\\
 & = \int_{\R^n \times \R^n}( c(x, y) -\l)d\gamma (x, y) + \l  \quad \hbox{(from \eqref{eq: prob})}. 
 \end{align*}
 This shows that the right-hand side of \eqref{eq: C lambda equivalence} is greater than or equal to the left-hand side. 
 For the other inequality of \eqref{eq: C lambda equivalence}, notice that for each $\gamma \in \Gamma_{\le}(f, g)$, there is a $\gamma' \in \Gamma(f, g)$ whose restriction to the set $
 \{ (x, y) \in \R^n \times \R^n \  |\  c(x, y) < \lambda\}$ is $\gamma$: in particular, one can take $\gamma' = \gamma + (f - f_\gamma) \otimes (g-g_\gamma)$, where $f_\gamma, g_\gamma$ are the marginals of $\gamma$. This completes the proof, including the claimed correspondence between the minimizers of \eqref{minprob} and \eqref{eq: lambda problem}. 
\end{proof}

\section{Preliminaries, assumptions, and further notation}\label{S: pre}
In this section, we explain relevant previous results, key assumptions of this paper, and further notation.

The existence and regularity of the optimal plan was done first by Caffarelli and McCann in \cite{Ca-Mc} under the additional hypothesis of the supports of $f$ and $g$ to be strictly separated by a plane. In \cite{Fi09, Fi10} Figalli is able to lift the restrictions and even further considers the case when $f$ and $g$ overlap on an open set. The strategies used in \cite{Ca-Mc} and \cite{Fi09, Fi10} differ greatly. In \cite{Ca-Mc} the authors used a Lagrange multiplier $\l$, add a point at infinity and study a full transport problem for a redefined cost function. The approach in \cite{Fi09, Fi10} is to study the convexity problems of the total cost $\mathcal{C}(m)$ of the optimizer with the given $m$,  and in this way solve the problem directly.  Our paper follows the approach of \cite{Ca-Mc}, especially because we treat $\l$ as a parameter of the partial transport problem.

\begin{assumption}\label{as: main}
 We now state basic assumptions we impose throughout the paper: 
 \begin{enumerate}[\bf 1.]
\item {\bf (cost function)} From now on, let $c(x, y) =\frac 1 2 |x-y|^2$ on $\R^n\times \R^n$. This is the most important special case of the more general class of examples (see Remark~\ref{rm: conditions on c}). We focus on this case for simplicity of treatment. 
 \item {\bf (domains)} $\Omega$, $\Lambda$ are smooth, open, connected and bounded domains.  \item {\bf (separation of source and target)} 
We assume that there is a hyperplane that separates $\Omega$ and $\Lambda$. 
 Moreover, we assume that there is a constant $\delta >0$ that $|x-y| \ge \delta$  for all $(x, y) \in \Omega \times \Lambda$.   
 
 \item  {\bf (measure $\sim$ volume)} $\mu= f$, $\nu = g $, and the functions $ f, g \in L^1$  satisfies $ \beta_1  \le f, g \le \beta_2$ for some constants $\beta_1, \beta_2 >0$ on $\Omega$, $\Lambda$, respectively, and $\spt f =\overline \Omega$, $\spt g = \overline \Lambda$.  
\end{enumerate}
\end{assumption}

\begin{remark}\label{rm: conditions on c}
One can generalize most of the results in this paper to the cost functions $c : \R^n \times \R^n \to \R$ that satisfy
\begin{enumerate}
 \item {\bf (smooth)} $c \in C^2$;
 \item {\bf (twisted)} the maps $ y \mapsto \nabla_x c(x, y)$ and $ x \mapsto \nabla_y c(x, y)$   are one-to-one for all $(x, y) \in \R^n \times \R^n$;
 \item {\bf (nondegenerate)} the mixed second order derivatives $[D^2_{xy} c (x, y)]$ give an invertible matrix.  
\end{enumerate}
One exception is the assertion 5 in Theorem~\ref{th: main}, where the strict monotonicity of free boundary movement is shown: for this result, we used the special structure of the cost function $c(x, y) =\frac 1 2   |x-y|^2$ in a essential way, to use the theory of classical Monge-Amp\`ere equation. 

 The above conditions on the cost function are the usual assumptions one require, especially in the regularity theory of optimal transport, though in the latter one require more assumptions theory, such as the Ma-Trudinger-Wang condition \cite{MTW05, TW09} (see also \cite{Loe09}). All these assumptions are satisfied by the quadratic cost $c(x,y) = \frac 1 2 |x-y|^2$ which is the main focus of many recent papers, including this paper.  
\end{remark}

\subsection{Minimizers, active regions,  and potential functions}\label{ss: min and pot}
This subsection is devoted to explain some of the fundamental contributions from Caffarelli and McCann \cite{Ca-Mc} that are relevant in our paper. 
As shown in \cite{Ca-Mc}, under Assumption~\ref{as: main},  there is a unique minimizer of \eqref{minprob}, and for
\begin{align}
 \gamma_\l := {\rm argmin} \, C_\l (f, g), 
\end{align}
 its mass satisfies the relation
\begin{align}\label{eq: m lambda}
m(\l):=\mathbf{m}(\gamma_\l)=-\frac{\partial C_\l(f,g)}{\partial\l}, 
\end{align}
moreover,  $m(\l)$ increases continuously from $0$ to $\min\{\|f\|_{L^1},\|g\|_{L^1}\}$. Therefore each mass $m$ can be attained for an appropriate value of $\l$.

To understand the behaviour of $\gamma_\l$, we let,  as in \cite{Ca-Mc}, $f_\l \le f, g_\l \le g$ denote the marginals of the minimizer $\gamma_\l \in \Gamma(f_\l , g_\l)$ of \eqref{minprob}. Then, the regions $\spt f_\l$, $\spt g_\l$ can be regarded as the (closure of the) {\em active regions} in the source and target domains, where the actual transport occurs. Therefore both the source and target domain can be decomposed as an active region and an inactive region, and the common boundary arises naturally as a free boundary of the partial transport problem. More precisely we have the following definition:
\begin{definition}
\begin{enumerate}[1.]
\item 
(Active regions)
   The {\em active regions }of the partial transport problem are given by the following:
\begin{align*}
 A^\Omega_\l &: =\{ x \in \Omega \ | \exists y,  (x, y) \in \spt \gamma_{\l}   \},\\ 
 A^\Lambda_\l &: = \{ y \in \Omega \ | \exists x,  (x, y) \in \spt \gamma_{\l}  \}
\end{align*}
Notice that   $\spt f_{\l}$,  $\spt g_{\l}$ are the closures of $A^\Omega_\l$, $A^\Lambda_\l$, respectively. 
When its meaning is clear from the context, we will simply use $A_i$ to denote the active regions.

\item (Free boundaries of the active regions)
Let 
\begin{align*}
 F_\l^\Omega = \partial A^\Omega_\l \setminus \partial \Omega, \quad
  F_\l^\Lambda = \partial A^\Lambda_\l \setminus \partial \Lambda. 
\end{align*}
\end{enumerate}
\end{definition}

One of the important observations made in \cite[Corollary 2.4]{Ca-Mc} (see also \cite[Proposition 3.1, Remark 3.3]{Fi10})  is the following interior ball condition for the active regions, which under Assumption~\ref{as: main} can be stated as
\begin{align}\label{eq: IBC}
 \spt f_\l &= \hbox{\rm closure of } \bigcup_{(x_1, y_1) \in \spt \gamma_\l } \{ x \in \Omega \  | \ c(x, y_1) < c (x_1, y_1)\},\\\nonumber
  \spt g_\l & =\hbox{\rm closure of }  \bigcup_{(x_1, y_1) \in \spt \gamma_\l } \{ y \in \Lambda \  | \ c(x_1, y) < c (x_1, y_1)\}.
\end{align}
Another important consequence from \cite[Corollary 2.4]{Ca-Mc} is that 
\begin{align}\label{eq: saturation}
 \hbox{ $f=f_\l$, $g=g_\l$ on $\spt f_\l$, $\spt g_\l$, respectively,}
\end{align}
 which is important throughout the paper. We also know from \cite[Theorem 3.4]{Ca-Mc} that the active regions increase monotonically along with the parameter $\l$  and therefore also with  $m$.  This will be an extremely important tool when comparing optimal plans associated to different $m$'s, $\l$'s: \begin{align}\label{eq: active monotone}
 &\spt f_{\l_1} \subset \spt f_{\l_2}, \quad \spt g_{\l_1} \subset \spt g_{\l_2} ;\\\nonumber
 & A^\Omega_{\l_1} \subset A^\Omega_{\l_2}, \quad A^\Lambda_{\l_1} \subset A^\Lambda_{\l_2} \quad \hbox{for $\l_1 \le \l_2$}. 
\end{align}

\subsubsection{\bf Augmentation with infinity}

 In \cite{Ca-Mc} Caffarelli and McCann introduced an effective way to treat the partial transport problem by adding an auxiliary point at infinity: Namely, the strategy is to attach a point $\hat\infty$ to $\R^n$ and extend the cost function $c$ 
\begin{align*}
\hat c(x,y)=\begin{cases}c(x,y)-\l & \mbox{if}\ x\neq\hat\infty \ \text{and} \ y\neq\hat\infty; \\
														0&\text{otherwise},
							 \end{cases}
\end{align*}
and the measures $d\mu=f(x)dx$ and $d\nu=g(y)dy$ to $\hat{\R}^n=\R^n\cup\{\hat\infty\}$ by 
\begin{align*}
\hat\mu=\mu+\|g\|_{L^1}\delta_{\hat\infty},\\
\hat\nu=\nu+\|f\|_{L^1}\delta_{\hat\infty}.
\end{align*}
Then, the problem  \eqref{minprob} is equivalent to 
\begin{align}\label{eq: hat problem}
 \inf_{\hat \gamma \in \Gamma(\hat \mu, \hat \nu)} \int_{\hat \R^n \times \hat \R^n} \hat c(x, y) d\hat \gamma (x, y)
\end{align}
where similarly as before $\Gamma(\hat \mu, \hat \nu)$ is the set of Borel measures on $\hat \R^n \times \hat \R^n$ with marginals $\hat \mu, \hat \nu$. In particular, the minimizer $\gamma_\l$ of \eqref{minprob} is nothing but the restriction of the minimizer $\hat \gamma_\l$ of \eqref{eq: hat problem} to $\R^n \times \R^n$. We refer \cite[Section 2]{Ca-Mc} for more details.

In particular, from \eqref{eq: saturation} this implies that points outside the active regions are matched with $\hat \infty$ by the plan $\hat \gamma_\l$. In other words, the difference $\hat \gamma - \gamma_
l$ is a measure supported on $$\left(\{  \spt (f-f_\l) \times \hat \infty \} \right)  \cup  \left(\{ \hat \infty \} \times (\spt (g-g_\l)\right) \cup \left(\{\hat \infty\} \times \{\hat \infty\} \right)$$ inside the product space $\hat \R^n \times \hat \R^n$. This  can be written as 
\begin{align}\label{eq: match with infty}
 \hat \gamma_\l -\gamma_\l 
&  \in \Gamma_{\le} \left(  f-f_\l,  \|f\|_{L^1} \delta_{\hat \infty}\right) \\\nonumber
& \ \ \ \ \ + \Gamma_{\le} \left( \|g\|_{L^1} \delta_{\hat \infty},  
  g-g_\l \right) + \Gamma_{\le}\left(\|g\|_{L^1} \delta_{\hat \infty} ,  \|f\|_{L^1} \delta_{\hat \infty} \right)
\end{align}
where $\Gamma_{\le} ( \alpha, \beta)$ denotes as before the set of measures whose marginals are bounded above by the  measures $\alpha, \beta$. 

\subsubsection{\bf Potential functions}
We will make use of the $\hat c$-potential functions in \cite{Ca-Mc}.  In our presentation, we will use $c$ or $\hat c$-convexity instead of $c$-concavity in \cite{Ca-Mc}, since convex functions are more natural in terms of Monge-Amp\`ere equations that we will consider later.

\begin{definition}[c-convexity and c-transform] \ 

\begin{enumerate}
 \item  A function $\hat{u}:\hat \R^n\to (-\infty, \infty]$ is said to be $\hat c$-convex if it is not identically $\infty$ on $\spt(\hat\mu)$ and satisfies
\[
\hat{u}(x)=\sup\limits_{y\in spt(\hat\nu)} - \hat{c}(x,y)-\hat{u}_{\hat{c}}(y)\equiv \hat{u}_{\hat{c}\hat{c}},
\]
where $\hat u^{\hat c}$ is the $\hat c$-conjugate:
\[
\hat{u}^{\hat{c}}(y)=\sup\limits_{x\in spt(\hat\mu)} - \hat{c}(x,y)-\hat{u}(x).
\]
\item  The $\hat c$-subdifferential of $\hat u$ is defined as 
 \begin{align*}
 &\partial^{\hat c} \hat u (x)   =\{ y \in  \hat \R^n   \ |\  \hat u (x) + \hat u_{\hat c} (y) = - \hat c(x, y)\},\\
 & \partial^{\hat c} \hat u = \{ (x, y) \in  \hat \R^n \times \hat \R^n   \ |\  \hat u (x) + \hat u_{\hat c} (y) =  - \hat c(x, y)\}.
\end{align*}
The $\hat c$-subdifferential $\partial^{\hat c} \hat u^{\hat c}$ of the $\hat c$-conjugate $\hat u^{\hat c}$ is similarly defined and from symmetry, we have
\begin{align*}
 \partial^{\hat c} \hat u = \partial^{\hat c} \hat u^{\hat c}.
\end{align*}
 \end{enumerate}
\end{definition}
\begin{remark}
 We obtain definitions for $\hat c$-concavity as in \cite[Definition 2.1]{Ca-Mc}, by just changing the sign and taking supremum instead of infimum above. 
\end{remark}
Note that the $\hat c$-convex functions are locally Lipschitz and semi-convex on $\R^n$, since they are obtained by taking supremum of $C^2$ functions (since $-\hat  c(x, y) = - \frac 1 2 |x-y|^2  +\l$ on $\R^n$). In particular, this implies  that $\hat u$   are differentiable a.e. and they are subdifferentiable, i.e. at each point $x \in \R^n$, the subdifferential 
\begin{align*}
 \partial \hat u(x) = \{  p \in \R^n \ | \ u(y) \le u(x) + p \cdot (y-x), \quad \forall y \in \R^n\}
\end{align*}
exists. 
This clearly holds also for $\hat u^{\hat c}$.  Recall that, from a well-known theorem of Alexandrov, semi-convex functions are twice differentiable a.e. 

The point of the $\hat c$-convex functions is the following fundamental fact:

\begin{lemma}[Uniqueness and characterization]\label{lem: unique} (See \cite[Lemma 2.3, Proposition 2.9]{Ca-Mc}.)
Use Assumption~\ref{as: main}. 
 There is a $\hat c$-convex function $\hat u$  (that is unique up to additive constants) such that for each optimizer $\hat \gamma$ of \eqref{eq: hat problem}, 
\begin{align}\label{eq: inside c-subdiff}
 \spt \hat \gamma \subset \partial^{\hat c} \hat u \,  (= \partial^{\hat c} \hat u^{\hat c}).
 \end{align}
\end{lemma}
\begin{remark}
 The above inclusion \eqref{eq: inside c-subdiff} can be improved, when restricted to $\Omega \times \Lambda$, to
\begin{align}\label{eq: equal to c-subdiff}
 \spt \hat \gamma \cap (\Omega \times \Lambda ) =  \partial^{\hat c} \hat u \cap (\Omega \times \Lambda ) , \quad \hbox{$\hat \gamma$-a.e.}.
\end{align}
The reason is that inside $\Omega\times\Lambda$, $\partial^{\hat c}\hat u$ is equal to the $c$-subdifferential $\partial^c u$ of the $c$-potential function $u$ (which is nothing but the restriction of the definition of $\partial^{\hat c} \hat u$ without having the point $\hat \infty$), which is well-known (see e.g. \cite{GM96}) to coincide $\hat \gamma$-a.e in $\Omega\times \Lambda$, with the graph of a Borel measurable map $T$, which then coincides with $\spt \gamma$. (The map $T$ is the optimal map in the ordinary optimal transport problem). 
 \end{remark}

\subsubsection{\bf Normalized potential $u_\l$} 
Since we will consider the optimization problem \eqref{minprob} (equivalently \eqref{eq: hat problem}, for each $\l$, let  $\gamma_\l$, $\hat \gamma_\l$, denote the minimizers  of \eqref{minprob}, \eqref{eq: hat problem}, respectively, and $\hat u_\l$, the corresponding $\hat c$-convex function for $\hat \gamma_\l$ given in Lemma~\ref{lem: unique}. 
Let $u_\l$ denote the restriction of $\hat u_\l$ to $\R^n$, and from now on, we use the normalization $$\hat u_\l (\hat \infty) =0.$$ 
For clarification, we call the restriction $u_\l$, {\em the normalized potential}. 
An important observation for normalized solutions is the following. Namely, from \eqref{eq: match with infty} and \eqref{eq: inside c-subdiff}, we see that 
\begin{align*}
\hat \infty \in  \partial^{\hat c} \hat u^{\hat c}_\l (y) \quad \hbox{ for each $y \in \hat \R^n \setminus A^\Lambda_\l$,}
\end{align*}
thus, for the normalized solution $\hat u_\l$, we have for each $y \in \R^n \setminus A^\Lambda_\l$, 
\begin{align*}
 \hat u_\l^{\hat c} (y) = - \hat c( \hat \infty, y)  - \hat u_\l (\hat \infty) = 0 + 0 = 0, 
\end{align*}
using the definition of the cost $\hat c$ and the $\hat c$-subdifferential.
Therefore, we see from the definition of $\partial^{\hat c} \hat u_\l$ and $u_\l$, 
\begin{align}\label{eq: u and m}
 u_\l  (x) &\ge  - c(x, y) + \l \quad \hbox{ whenever $x\in \R^n$,  and $y \in \hat \R^n \setminus A^\Lambda_\l$;}\\\nonumber
 u_\l  (x) & =  - c(x, y) + \l  \hbox{ if moreover $ y \in \partial^{\hat c} \hat u (x)$. }
\end{align}

\subsubsection{\bf $C_{loc}^\alpha$ regular partial transport mappings}
We finish this section with a characterization and regularity result given in \cite{Ca-Mc}:
\begin{theorem}[Caffarelli and McCann \cite{Ca-Mc}; see also \cite{Fi10, Fi09}]\label{th: Ca-Mc}
Use Assumption~\ref{as: main}. Then, for each $0< \l < \infty$, the optimal transport $\gamma_\l$, i.e. the solution to \eqref{minprob}, uniquely exists and is given by a $C^\alpha_{loc}$ mapping $ T_\l: A^\Omega_\l \to A^\Lambda_\l$ 
such that 
\begin{align*}
\gamma_\l = (id \times T_\l)_\# f_\l
\end{align*}
and 
\begin{align*}
 T_\l (x)  = x + \nabla u_\l (x). 
\end{align*}
Here, $\alpha>0$ depends on $n, f, g$. 
Moreover, the free boundaries $F^\Omega_\l, F^\Lambda_\l$ are $C^1_{loc}$. 
\end{theorem}

\section{Main results}\label{S: main}
We now explain in detail our main results. 
First, we introduce some notation for simplicity of the presentation. 

\begin{definition}[Notation $\lesssim, \gtrsim, \sim$]
 We use  $a\gtrsim b$, $a\lesssim b$, $a\sim b$, to denote
$a \ge  C_1 b$, $a \le C_2 b$, $a\ge C_1 b  \quad \& \quad a \le C_2b$, respectively, for some constants $C_1, C_2$, depending only on $\Omega, \Lambda$ and the lower and upper bounds of $\mu$, $\nu$ on $\Omega$, $\Lambda$, respectively and in particular, not on the parameters $\lambda$, $m$. 

\end{definition}

\subsection{Main question}
The main aim of this work is to get quantitative results on how the active region changes under the  variation of the parameters $m$ or $\l$. 
We give further notation and a few definitions to set up the question more precisely, 
\begin{definition}[Notation and definitions]
\begin{enumerate}[1.]
\item 
For $i=1,2$, consider the values $\lambda_i$, $m_i = m(\l_i)$ with 
$0 < \lambda_1 < \lambda_2$
$0 < m_1 < m_2<1$. 
We let $\gamma_\l$, $\hat \gamma_i$, $u_i$, $\hat u_i$, $i=1,2$, denote the corresponding objects $\gamma_\l$, $\hat \gamma_\l$, $u_\l, \hat u_\l$ in the previous section. By the same way, we  let $A^\Omega_i$, $A^\Lambda_i$, $F^\Omega_i$, $F^\Lambda_i$, denote the corresponding active regions and free boundaries.  When its meaning is clear from the context, we will simply use $A_i$, $F_i$.

\item  ($b$-distance) We define the $b$-distance between the free boundaries of the active regions (in $\Omega$): $$\distb (F_1, F_2):= \inf\{\epsilon \ | \ F_j \subset N_\epsilon (F_i \cup \partial \Omega), \quad i, j = 1, 2 \}  $$
This distance is similar to the Hausdorff distance, but it  considers the effect from the boundary $\partial \Omega$. 
\end{enumerate}
 
\end{definition}

\begin{remark}
 Note that smallness of $\distb(F_1, F_2)$ implies the two free boundaries $F_1$ and $F_2$ are close to each other in a uniform manner outside a small neighbourhood of the boundary of the domains. Incorporating the boundary $\partial \Omega$ for measuring the distance between the free boundaries looks somewhat  technical,  but our method in the proof of Theorem~\ref{th: main} assertion 3 requires it. However, 
one can view this distance $\distb$  as  the Hausdorff distance  between the boundaries of the inactive regions $\Omega \setminus A_i$, $i=1,2$. 
\end{remark}

We now use the above notation to state the main question of the present paper:
\begin{question}[Main question]\label{qe: main} \ 
\begin{itemize}
 \item Can we estimate $\distb (F_1, F_2) $  in term of the values $m_1, m_2$ or $\lambda_1, \lambda_2$?
 \item Since clearly the cost cap $\l$ and the mass $m$ are related, can we find a quantifiable relation between them?
\end{itemize}
\end{question}

\subsection{Main Results}

Addressing Question~\ref{qe: main}, our main results give quantitative estimates for the free boundary movement as the parameters $m$ and $\lambda$ vary. 

\begin{theorem}[Quantitative results on free boundaries]\label{th: main}
Use Assumption~\ref{as: main}.
Then, the following hold:
\begin{enumerate}
 \item[1.] {\bf (Lower bound of speed of free boundary movement)} 
 $$m_2 - m_1 \lesssim \distb (F_1, F_2);$$
 \item[2.] {\bf (Holder continuity of free boundary movement  in $m$)} 
 $$\distb (F_1, F_2) \lesssim (m_2 - m_1)^{\frac{1}{n}};$$
 
 \item[3.] {\bf (Lipschitz bound of free boundary  movement in $\lambda$)} 
 $$\distb (F_1, F_2) \lesssim (\lambda_1 - \lambda_1);$$
 
 \item[4.] {\bf (Lipschitz bound of $m(\lambda)$)}  
 $$m_2-m_1 \lesssim \lambda_2-\lambda_1.$$
 
 \item[5.] {\bf (Strict monotonicity of free boundary movement)} 
  Assume further that $f, g  \in C^\alpha$ for some $\alpha>0$, then $$F_1 \cap F_2 =\emptyset .$$ 
\end{enumerate}
\end{theorem}

\begin{proof}
 The assertion 1 and 2 will be shown in Section~\ref{S: mass and free}. The assertion 3 will be shown in Section~\ref{S: cost and free}. The assertion 4 follows  immediately from the assertions 1 and 2. Finally, the assertion 5 will be shown in Section~\ref{S: strict mono}. 
\end{proof}

An immediate and interesting consequence of Theorem~\ref{th: main} Assertion 4 together with \eqref{eq: m lambda}  is this:
\begin{corollary}\label{cor: Lip in lambda}
Use Assumption~\ref{as: main}. 
The cost function $C_\l$ given by
\[
C_\l(f,g)=\inf\limits_{\gamma\in\Gamma_\leq(f,g)}\int\limits_{\R^n\times\R^n}[c(x,y)-\l]d\gamma(x,y),
\]
is a $C^{1,1}$ function as a function of $\l$.
\end{corollary}
\begin{remark}
 Note as observed in \cite{Ca-Mc} that the function $\lambda \mapsto C_\l (f, g)$ is concave because it is an infimum of  a family of linear functions in $\l$. 
\end{remark}

Notice that the H\"older exponent $1/n$ in the assertion 2 in Theorem~\ref{th: main} is sharp, as seen the example when the target measure is given by (an approximation of) a Dirac mass, when $m$ is close to $0$. 

\begin{remark}[Dirac delta target]\label{rm: dirac}

Consider the case that $\mu$ is just the uniform density in $B_1(0)$ and $\nu$ is Dirac delta concentrated in $2e_n$. We note that the free boundary is always of the form
\[
F=(\cup_{x\in A} B_{r(x)}(x))\cap \Omega.
\]
Therefore, since the target region is just one point we conclude that the active region is given by $B_1(0)\cap B_r(x)$. We note from this example several remarkable properties.
\begin{itemize}
\item[(a)] The free boundary is strictly monotone.
\item[(b)] The separation of the free boundary can be controlled directly by the mass. More precisely we have that the free boundaries are withing a tubular neighborhood of size $m^{1/n}$. Let $0<m_1<m_2\leq1$, then the associated free boundaries for the partial optimal transport with masses $m_1$ and $m_2$ are given by
\[
\distb (F_1,F_2)\lesssim(m_2-m_1)^{1/n}.
\]
\end{itemize}

\end{remark}

However, we have the following conjecture:
\begin{conjecture}\label{conj: main}
 Use the same assumptions and notation as Theorem~\ref{th: main}. We conjecture the following reinforcement of Theorem~\ref{th: main}, which are all related to each other: 
\begin{enumerate}
  \item $m_2 - m_1  \sim \lambda_2-\lambda_1$ for  $ \frac{1}{3} \le  m_i \le \frac{2}{3}$;
  \item $\distb (F_1, F_2) \sim m_2 - m_1$ for $ \frac{1}{3} \le  m_i \le \frac{2}{3}$ ;

 \item For $ \frac{1}{3} \le  m_i \le \frac{2}{3}$,   
 there exists $\delta$ such that $\delta >0$ with $\delta \sim m_2 - m_1$ and  $N_\delta (F_1) \cap N_\delta (F_2) =\emptyset$.
\end{enumerate}
\end{conjecture}
This conjecture predicts that the movement of the free boundary and the growth of the mass are linear in $\lambda$. Here, the condition $ \frac{1}{3} \le  m_i \le \frac{2}{3}$ is required as it can be easily seen from the case where $\nu$ is close to a Dirac mass $\delta_{y_0}$ and when $m$ is close to $0$ or $1$: see Remark~\ref{rm: dirac}. The assertion 5 of Theorem~\ref{th: main} can be regarded as a partial result for  the item (3).

\section{Mass  $m$ and the free boundary: Proof of Theorem~\ref{th: main}, assertions 1 and 2}\label{S: mass and free}

We now start proving  the assertions in Theorem~\ref{th: main}.  In this section, we focus on the free boundary in the source domain $\Omega$. Because the assumptions we made are symmetrical for $\Omega$ and $\Lambda$, exactly the same proof shows the statement for the free boundary in $\Lambda$.

We first show the assertions 1 and 2. These  follow easily from the property \eqref{eq: IBC} and  simple geometric arguments.
 First, observe the following fact:

\begin{lemma}[Free boundary Lipschitz and semi-convex]\label{lem: F Lip and semi-convex}
Use Assumption~\ref{as: main}. 
 The free boundary $F$ is uniformly Lipschitz and semi-convex. The Lipschitz and semi-convexity constants depend only on $\Omega$ and $\Lambda$. 
In particular, $\mathcal{H}^{n-1} (F)$ is uniformly bounded. 
\end{lemma}
\begin{proof}
 It follows immediately from 
 Assumption~\ref{as: main} (especially,  the items 2 and 3) and the interior ball property \eqref{eq: IBC}. The uniform boundedness of $\mathcal{H}^{n-1}(F)$  is due to Lipschitz property of $F$ and boundedness of $\Omega$). 
\end{proof}

\begin{proof}[\bf Proof of Assertion 1 of Theorem~\ref{th: main}]
 From Assumption~\ref{as: main} item 4, to get upper bound on the difference $m_2-m_1$, we only need to estimate the volume difference between the two active regions $A_1$ and $A_2$. From the Lipschitz property of $F_1$ and $F_2$ due to Lemma~\ref{lem: F Lip and semi-convex}, the volume difference is estimated as
\begin{align*}
\vol (A_2 \setminus A_1) \lesssim  \left[ \mathcal{H}^{n-1} (F_1) + \mathcal{H}^{n-1} (F_2)\right] \distb (F_1, F_2)
\end{align*}
which then implies 
\begin{align*}
 m_2 - m_1 \lesssim \left[ \mathcal{H}^{n-1} (F_1) + \mathcal{H}^{n-1} (F_2)\right] \distb (F_1, F_2).
\end{align*}
From the uniform boundedness of $\mathcal{H}^{n-1}(F)$  (Lemma~\ref{lem: F Lip and semi-convex}), we get 
\begin{align*}
 m_2 - m_1 \lesssim \distb(F_1, F_2).
\end{align*}
This proves the assertion 1. 
\end{proof}

\begin{proof}[\bf Proof of Assertion 2 of Theorem~\ref{th: main}]

For the two free boundaries $F_1$ and $F_2$, let $x \in F_2$  be the point realizing the distance $d= \dist(F_1, F_2)$, i.e. $B_d(x) \cap F_1 =\emptyset$.  Thanks to the interior ball condition  \eqref{eq: IBC}, then there exists $y \in \Lambda_{A_2}$ such that $B_{\dist(x, y)} (x) \cap \Omega \subset A_2$.
This implies 
\begin{align*}
 B_d (x) \cap B_{\dist(x, y)} \cap \Omega \subset A_2 \setminus A_1.
\end{align*}
Note that 
\begin{align*}
\vol\left(  B_d (x) \cap B_{\dist(x, y)} \cap \Omega \right) \sim d^n, \quad \vol(A_2 \setminus A_1) \sim m_2 - m_1.
\end{align*}
Here we used Assumption~\ref{as: main}. 
This shows
\begin{align*}
 \distb(F_2, F_1) \lesssim (m_2-m_1)^{1/n}
\end{align*}
as desired. 
\end{proof}

\section{Monotonicity of the potential functions}\label{S: monotone potential}

 For the assertion 3 of Theorem~\ref{th: main}, we need monotonicity on the potential $ u_\lambda$  with respect to $\lambda$. 
 The goal of this section is to prove such monotonicity: 
 \begin{theorem}[Monotonicity]\label{th: monotonicity}
Use Assumption~\ref{as: main}. 
Let $u_{\l_1}$ and $u_{\l_2}$ be the normalized solutions of the partial optimal transport problem \eqref{eq: lambda problem} associated to $\l_1$ and $\l_2$, respectively, 
Then $u_{\l_1} \leq  u_{\l_2}$ in $\Omega$. 
\end{theorem}
For  the sake of a clearer exposition from now on we let 
\begin{align}\label{eq: v i}
  v_i (x) = u_{\lambda_i} (x) + \frac 1 2 |x|^2. 
\end{align}
An important observation is that $v_i$'s are convex for $c$-convex functions $u_i$'s (here $c(x, y) = \frac 1 2 |x-y|^2$). Moreover, $\partial v_i (x) = \partial^c u_{\lambda_1} (x)$, thus we can regard $\nabla v_i$ as the transport map for each partial transport problem \eqref{eq: lambda problem} with $\l_1, \l_2$, respectively. Here, each $\nabla v_i$, $i=1, 2$, is viewed as a measurable mapping  defined a.e. on $\R^n$.  
Now, let us recall a version of the Aleksandrov Lemma as stated in \cite{McC95}.
\begin{lemma}[Alexandrov's lemma]\label{lem: Alex}
Let $v_1, v_2 : \R^n \to \R$ be convex functions, differentiable at a point $p$ with $v_1(p)=v_2(p)$, and $\nabla v_1 (p)\ne \nabla v_2 (p)$. Define $M=\{v_1>v_2 \}$ and $X=\{\nabla v_2^{-1}(\partial v_1(M))\}$. Then $X\subset M$, while $p$ lies at a positive distance from $X$.
\end{lemma}

We use this lemma to show Theorem~\ref{th: monotonicity}:
\begin{proof}[\bf Proof of Theorem~\ref{th: monotonicity}]
Recall the definition of the convex functions $v_i$ in \eqref{eq: v i}  and note that it suffices to show $v_1 \le v_2$ in $\Omega$. 

We will proceed by contradiction. Suppose that $v_1$ and $v_2$ cross in $\Omega$, i.e. $M=\{ v_1 > v_2\}\ne \emptyset$  . We first assume  that they cross in a transversal fashion, namely, there is a point $p \in \Omega$ such that $v_1(p)=v_2(p)$ and $v_1, v_2$ are differentiable at $p$ with  $\nabla v_1(p)\neq \nabla v_2(p)$. 
 We note that $p\in\partial M$, $M$ is a subset of the active region of $A_1$ for $\lambda_1$.  Let $Y=\partial  v_1 (M)$.
Due to Lemma~\ref{lem: Alex} we know there is an open neighborhood of $p$, $\mathcal{N}_p\subset\Omega$ such that $\nabla v_2^{-1}(Y)\subset M $ and excludes $\mathcal{N}_p$. This implies 
\[
\mu((\nabla v_2)^{-1}(Y))<\mu(M)\leq \mu((\nabla v_1)^{-1}(Y)).
\]
The first inequality comes from the fact that $\mu(M\cap\mathcal{N}_p)>0$ and the neighborhood $\mathcal{N}_p$  is not included in $\nabla v_2^{-1}(Y)$. The second inequality follows from the fact that $v_1$ is differentiable a.e., so $\mu (M \setminus (\nabla v_1)^{-1}(Y) )=0$.  Since $Y\subset A^\Lambda_1\subset A^\Lambda_2$ we get a contradiction,  as both $\nabla v_1$ and $\nabla v_2$ push forward $\mu$ into $\nu$ in that region $Y$. 

In the case that $v_1$ and $v_2$ cross in $\Omega$,  but does not  intersect transversally in the interior, we proceed as follows. 
Lower the function $v_1$ down a bit, i.e. consider $v_1^\epsilon=v_1 - \epsilon$ for each small $\epsilon>0$  (here, $v_1^\epsilon$ is not normalized anymore) and note that the active region and the mass moved by $\nabla v_1^\epsilon$ is the same as for $v_1$, because they depend only on $\nabla v_1$.  Since the crossing was in the interior, then for small $\epsilon$, $v_1^\epsilon$ and $v_2$ will still  cross in the interior of $\Omega$. Consider the family $\{v_1^\epsilon\}_{\epsilon >0}$. First, notice that  
\begin{align*}
 \cup_{\epsilon>0}  \{ v_1^\epsilon = v_2\} \cap \{ v_1 > v_2 \} =  \{ v_1 > v_2 \}. 
 \end{align*}
Suppose each set $\{ v_1^\epsilon = v_2\}\cap \{ v_1 > v_2 \} $   has no point where the crossing between $v_1^\epsilon$ and $v_2$ is transversal. Then, it implies that  $\nabla v_1 = \nabla  v_2$ a.e. in $\{ v_1 > v_2\}$. This means $v_1 = v_2$ in $\{ v_1 > v_2\}$, an obvious contradiction. So, for some $\epsilon>0$, $v_1^\epsilon$ and $v_2$ cross transversally in the interior of $\Omega$. 
Finally, we can apply the same reasoning as before, to $v_1^\epsilon$ and $v_2$, to derive a contradiction.
\end{proof}

\section{Cost cap $\lambda$ and the free boundary: Proof of Theorem~\ref{th: main}, assertions 3}\label{S: cost and free}

In this section, we prove the following proposition, which verifies the assertion 3 in Theorem~\ref{th: main}.
\begin{proposition}\label{prop: dist F1 F2}
Use Assumption~\ref{as: main}. 
Recall  the free boundaries $F_i^\Lambda$  to the solutions of  \eqref{minprob}, with $\l_1<\l_2$, respectively.  Then
\begin{align}\label{eq: F2 F1}
 &F^\Lambda_2 \subset N_{C(\lambda_2-\lambda_1)} (F_1 \cup \partial \Lambda),\\\nonumber
 &F^\Lambda_1 \subset N_{C(\lambda_2-\lambda_1)} (F_2 \cup \partial \Lambda),
\end{align}
for some constant $C$, depending only on $\Omega, \Lambda$. Note that this implies 
\begin{align*}
 \distb(F^\Lambda_1, F^\Lambda_2) \lesssim \lambda_2 -\lambda_1.
\end{align*}
Since Assumption~\ref{as: main} is symmetrical, we also have 
\begin{align*}
 \distb(F^\Omega_1, F^\Omega_2) \lesssim \lambda_2 -\lambda_1.
\end{align*}
\end{proposition}
\begin{proof} 
Let $u_1$ and $u_2$ be the normalized potential functions to the solutions of  \eqref{minprob}, with $\l_1<\l_2$ respectively.

We begin by showing the first inclusion in  \eqref{eq: F2 F1}.
Consider an arbitrary point $y_2 \in F_2^\Lambda$. 
It has a corresponding point $x_2 \in {\rm cl} \Omega$ (in fact in $\partial \Omega$ due to \cite[Corollary 6.9]{Ca-Mc}), such that $y_2 \in \partial^{\hat c}  u_2 (x_2)$.  Consider the line $\overline{x_2 y_2}$, and then either $\overline{x_2 y_2} \cap F_1^\Lambda =\emptyset$ or there exists $y_1 \in \overline{x_2y_2}\cap F_1^\Lambda$. In the first case, there is a point in  $ \overline{x_2 y_2} \cap \partial \Lambda \cap \left(\R^n \setminus A^\Lambda_1\right)$, and we denote it also by $y_1$. Notice that the functions $m_i(x)=-\frac 1 2  |x-y_i|^2+\l_i,$ $i=1,2$, satisfy $m_i \le u_i$, and $m_2(x_2) = u_2(x_2)$, since we can apply \eqref{eq: u and m} to $y_2 \in \partial^{\hat c} u_2 (x_2)$ and $y_i \in \subset \R^n \setminus A^\Lambda_i$. 
Therefore, noting that $u_1 \le u_2$ in $\Omega$ from Theorem~\ref{th: monotonicity},
$
 -c(x_2,y_2)+\l_2 \ge -c(x_2,y_1)+\l_1,
$
which gives us
\[
c(x_2,y_2)-c(x_2,y_1)\le \l_2-\l_1.
\]
Since $c(x,y) = \frac 1 2  |x-y|^2$ and $\Omega$ and $\Lambda$ are  separated (3 of Assumption~\ref{as: main}), using the fact the points $x_2, y_1, y_2$ are on the line $\overline{x_2 y_2}$, 
we have then $$c(x_2,y_2)-c(x_2,y_1)\gtrsim  |y_2-y_1|,$$ proving the desired result  $F^\Lambda_2 \subset N_{C(\lambda_2-\lambda_1)} (F^\Lambda_1 \cup \partial \Omega)$.

By interchanging the role of $F^\Lambda_2$ and $F^\Lambda_1$, the exactly same method shows the other inclusion $F^\Lambda_1 \subset N_{C(\lambda_2-\lambda_1)} (F^\Lambda_2 \cup \partial \Omega)$, completing the proof. 
\end{proof}

\section{Strict monotonicity of free boundaries: Proof of Theorem~\ref{th: main}, assertion 5}\label{S: strict mono}

A desired result on the free boundary movement is to get a quantitative separation of free boundaries, showing points in two different free boundaries corresponding two different values of $m$ or $\lambda$, are away from each other by a positive distance whose lower bound is controlled by $m$ or $\lambda$. As a partial progress toward this direction, we show in this last section, the free boundaries indeed do separate from each other, but without a quantitative estimate,  which also establishes the strictness of the monotonicity of the active regions. 
To do this we require the following technical assumptions, which are mainly for using the established regularity theory of the optimal partial transport \cite{Ca-Mc}:

\begin{assumption}[Additional assumptions for strict monotonicity]\label{as: tech} \ \ \ \
\begin{enumerate}[1.]
 \item {\bf (smooth and convex domains)} $\Omega$, $\Lambda$ are strictly convex and bounded domains with smooth boundaries. 
 \item  {\bf (smooth densities)}  
 The functions $ f, g \in C^1$.
 \end{enumerate}
\end{assumption}

For $f, g \in C_{loc}^\alpha$, consider $v(x) = u(x) + \frac 1 2 |x|^2$, where $u$ is the potential function to the solution of  the partial transport problem \eqref{minprob}. Then, $v$ satisfies in the interior of the active region, $\det(D^2 v ) = f(x) / g(\nabla v )$,and $v$ (thus $u$)  is $C_{loc}^{2,\alpha}$; this interior $C^2$ regularity is due to Caffarelli \cite{Caf90-W2}. The reason we require a further regularity $C^1$ is to apply the Hopf lemma (Lemma~\ref{lem: Hopf}) and the strong maximum principle  (Lemma~\ref{lem: strong max}) shown below using the $C^1$ assumption on $f$ and $g$.

\subsection{Hopf's lemma and the strong maximum principle}
In this subsection for reader's convenience, we give a short proof of a Hopf Lemma and strong Maximum Principle for the Monge-Amp\`ere equation
\begin{align}\label{eq: MA}
 \det(D^2 v (x))=F(x, \nabla v (x)),
\end{align} 
which should be known to experts. We assume that $F$ is $C^1$, which for our partial transport problem correspond to the case where the densities $f$,$g$ of the measures $\mu$, $\nu$, respectively, are $C^1$, and $g$ is bounded from below. 

\begin{lemma}[Hopf lemma for Monge-Amp\`ere]\label{lem: Hopf}
Let $B$ be an open ball and 
 $v_1, v_2\in C^2(B)\cap C^1(\overline{B})$ be uniformly convex solutions of \eqref{eq: MA} in $B$, where $F$ is a differentiable function. Suppose $v_2> v_1$ on $B$ and there is a point $x_0 \in \partial B$ such that $v_1(x_0)=v_2(x_0)$. Then, we have the strict inequality
\[
\frac{\partial v_2}{\partial \vec n} (x_0)> \frac{\partial v_1}{\partial \vec n}(x_0).
\]
where  $\vec n$ is the inner normal vector of $\partial B$ and $x_0$. 
\end{lemma}

\begin{lemma}[Strong maximum principle for Monge-Amp\`ere]\label{lem: strong max}
Let $U$ be an open connected domain and 
 $v_1, v_2\in C^2(U)$ be uniformly convex solutions of \eqref{eq: MA}  in $U$, where $F$ is a differentiable function. Suppose $v_2\geq v_1$ on $U$ and there is a point $x_0 \in U$ such that $v_1(x_0)=v_2(x_0)$, then $v_1=v_2$ on $U$. 
\end{lemma}
\begin{proof}[\bf Proof of Lemmata~\ref{lem: Hopf} and \ref{lem: strong max}] 
Following a standard argument below, we reduce to the classical Hopf's lemma and the strong maximum principle of the liner uniformly elliptic equations. 
First we note that since the solutions are $C^2$, then the Monge-Ampere equations are uniformly convex (non-degenerate). 
We observe that for $w=v_1-v_2$, 
\begin{align*}
0 & = \det(D^2v_1)-\det(D^2v_2) -  (F(x, \nabla v_1(x) ) - F(x, \nabla v_2(x))) \\
&=\int_0^1\frac{d}{dt}\left[\det((D^2(tv_1+(1-t)v_2)(x)) + F(x, \nabla ((1-t)v_1 + tv_2))(x))\right]dt\\
&=a_{ij}D^2_{ij}w + c_j D_j w
\end{align*}
where
\begin{align*}
 a_{ij} &= [\det D^2 (tv_1+ (1-t)v_2)] \left[D^2 (tv_1+ (1-t)v_2)\right]^{-1}_{ij} \\
 c_i (x)  &= \left[D_{p_i} F \right](x, \nabla (1-t) v_1(x) + t v_2 (x)).
\end{align*}
Here, $D_{p_i} F$ is the derivative of $F(x, p)$ in the $p$ variable. 
We notice that  the coefficients $a_{ij}$  are uniformly elliptic $U$ due to uniform convexity and the $C^2$ regularity of $v_1$, $v_2$. 
Now, one can use the classical  Hopf's lemma  and the strong maximum principle for liner elliptic equations (see e.g. \cite[Theorem 3.5]{GT01}). 
\end{proof}

\subsection{Strict monotonicity of the active regions. }
Suppose we have two Lagrange multipliers $\l_2>\l_1$. We would like to study the strict monotonicity of the active regions. We already know that they are all monotone, but at this time we can't discard points for which the free boundary for different Lagrange multipliers remain fixed.

To achieve this we approach the problem in two steps. The first step is to note that if we have a point that belongs to both free boundaries, then the image through each optimal map is the same; here, we use the $C^1$ regularity of $u$ along the free boundary from \cite{Ca-Mc}. The second step is to use the monotonicity of the solutions plus a Hopf Lemma argument to contradict this fact.
For the discussion of this section, recall the notation $ v_i (x) = u_i (x) + \frac 1 2 |x|^2. $

\begin{lemma}[Touching implies the same gradient]\label{lem: touching} Under Assumptions~\ref{as: main} and \ref{as: tech}, 
let $\l_2>\l_1$ and suppose that there is $x_0$ such that $x_0\in \partial F_1$ and $x_0\in \partial F_2$. Then we have $\nabla v_1(x_0)=\nabla v_2(x_0)$.
\end{lemma}
\begin{proof}
 Note that from Theorem~\ref{th: Ca-Mc},  we have that the optimal partial transport maps are given by $F_i = \nabla v_i$. Now, let us first consider the case 
\[
\frac{\nabla v_1 (x_0)}{|\nabla v_1 (x_0)|} \ne \frac{\nabla v_2 (x_0)}{|\nabla v_2 (x_0)|}.
\]
Then we have two supporting balls with different centers and radii, namely $B_{r_1}(\nabla v_1(x_0))$ and $B_{r_2}(\nabla v_2(x_0))$, whose boundaries cross at  $x_0$ transversally. Now, note that a graph satisfying the previous statement cannot be $C^1$, since no $C^1$ graph can have two supporting balls.  Therefore, since the free boundaries are $C^1$, as shown in \cite{Ca-Mc}, this leads to a contradiction. 

Now suppose we are in the second scenario
\[
\frac{\nabla v_1 (x_0)}{|\nabla v_1 (x_0)|} = \frac{\nabla v_2 (x_0)}{|\nabla v_2 (x_0)|}.
\]
 Note that this implies that the image $\nabla v_1(x_0)$ and $\nabla v_2(x_0)$ lie in the same line from the point $x_0$. But, notice that from \cite[Corollary 6.9]{Ca-Mc}, they are both on the boundary $\partial \Lambda$. Thus, strict convexity of $\Lambda$ implies $\nabla v_1(x_0) = \nabla v_2(x_0)$,  completing the proof. 
\end{proof}

We now prove the desired strict monotonicity of the active regions, by showing no intersection of the free boundaries  occurs for different values of $\lambda$. 
\begin{theorem}[Strict monotonicity of the active regions]\label{th: strict free}
Under Assumptions~\ref{as: main} and \ref{as: tech}, 
let $\l_2>\l_1$. Then we have that $F_2\cap F_1\cap \Omega =\emptyset$.
\end{theorem}
\begin{proof}
First, we suppose that there is  an interior point in the free boundary that belongs to both free boundaries, that is, there is $x_0\in F_2\cap F_1\cap \Omega$. Due to Lemma~\ref{lem: touching}, we know that $\nabla v_2 (x_0 ) = \nabla v_1 (x_0)$. Now, due to the monotonicity of $u_2$ and $u_1$ (Theorem~\ref{th: monotonicity}) we can apply the strong maximum principle (Lemma~\ref{lem: strong max})  in any region in $A_1\subset A_2$, where both functions are $C^2$, showing  that $u_2 > u_1$ (thus $v_2 > v_1$) in $A_1$. 
Finally, notice that at the point $x_0$, the free boundary $F_1$ satisfy the interior ball condition \eqref{eq: IBC}, thus we can apply Hopf Lemma (Lemma~\ref{lem: Hopf}) to derive a contradiction. This completes the proof. 
\end{proof}

\bibliography{mybiblio-dynamicsOT}
\bibliographystyle{plain}

\end{document}